\documentclass[journal,onecolumn,10pt]{IEEEtran}
\usepackage{amsmath,amssymb,dsfont,graphicx,color,lmodern}
\usepackage{amsthm}
\usepackage[T1]{fontenc}
\usepackage[utf8]{inputenc}
\usepackage{amsmath,amssymb,euscript,yfonts,psfrag,latexsym,dsfont,graphicx}
\usepackage{bbm,color,amstext,wasysym,parskip,balance,cite}

\usepackage{bm}
\usepackage{subfigure}

\newcommand{\Id}{\mathrm{Id}}
\newcommand{\ud}{\frac{1}{2}}

\newcommand\x{\mathbf{x}}
\newcommand\y{\mathbf{y}}

\newcommand{\Rd}{\mathbf{R}^d}

\newcommand\PD{\mathcal{P}_2(\Rd)}
\newcommand\PDD{\mathcal{P}_2(\mathbf{R}^{2d})}

\newcommand{\bol}{\mathbf}

\newtheorem{theo}{Theorem}
\newtheorem{prop}[theo]{Proposition}

\newtheorem{rem}[theo]{Remark}

\newtheorem{defi}[theo]{Definition}
\newtheorem{problem}[theo]{Problem}

\newcommand{\be}{\begin{equation}}
\newcommand{\ee}{\end{equation}}
\def\benu{\begin{enumerate}}
\def\eenu{\end{enumerate}}

\newcommand{\beq}{\begin{equation}}\newcommand{\eeq}{\end{equation}}
\title{\LARGE \bf
Statistical learning in Wasserstein space
}

\author{Amirhossein Karimi, Luigia Ripani, and Tryphon T.\ Georgiou
\thanks{Amirhossein Karimi and Tryphon T.\ Georgiou are with the Department of Mechanical and Aerospace Engineering, University of California,
Irvine, CA, USA
        {\tt\small amirhosk@uci.edu, tryphon@uci.edu}}%
\thanks{Luigia Ripani is with the University of Cergy-Pontoise, France
        {\tt\small luigia.ripani@u-cergy.fr}}%
\thanks{Supported in part by the NSF under grants 1807664, 1839441, and the AFOSR under FA9550-17-1-0435.}
}

\graphicspath{{./},{./pictures/}}
\def\spacingset#1{\def\baselinestretch{#1}\small\normalsize}
\spacingset{1}

\begin{document}

\maketitle
\thispagestyle{plain}
\pagestyle{plain}

\begin{abstract}
We seek a generalization of regression and principle component analysis (PCA) in a metric space where data points are distributions metrized by the Wasserstein metric. We recast these analyses as multimarginal optimal transport problems. The particular formulation allows efficient computation, ensures existence of optimal solutions, and admits a probabilistic interpretation over the space of paths (line segments). Application of the theory to the interpolation of empirical distributions, images, power spectra, as well as assessing uncertainty in experimental designs, is envisioned.
\end{abstract}

\section{INTRODUCTION}

The purpose of statistical learning is to identify structures in data, typically cast as relations between explanatory variables. Prime example is the case of principal components where affine manifolds are sought to approximate the distribution of data points in a Euclidean space. The data is then seen as a point-cloud clustering about the manifolds which, in turn, can be used to coordinatize the dataset. Such problems are central to system identification \cite{kalman1982identification,ljung2010perspectives}.

Accounting for measurement noise and uncertainty has been an enduring problem from the dawn of statistics to the present \cite{kalman1982identification}. In this work we view distributions as our principle objects which, besides representing point clouds, may also represent the uncertainty in localizing individual elements in a point set. Our aim is to develop a mathematical framework to allow for an analogue of principle components in the space of distributions. That is, we view distributions as points themselves in a suitable metric space (Wasserstein space) and seek structural features in collections of distributions. 

From an applications point of view, these distributions may delineate density of particles, concentration of pollutants, fluctuation of stock prices, or the uncertainty in the glucose level in the blood of diabetic persons, indexed by a time-stamp, with the goal to capture the principal modes of underlying dynamics (flow). Another paradigm may involve 
intensity images in which case we seek suitable manifolds of distributions that images cluster about. Such manifolds would constitute the analogue of principal components and can be used to parametrize our image-dataset and capture directions with the highest variability. Such a setting could be applied in e.g., in longitudinal image study~\cite{niethammer2011geodesic} where the brain development or tumor growth patterns need to be studied, power spectral tracking~\cite{jiang2011geometric}, traffic control~\cite{hong2014geodesic} and so on. 

There is a very recent and fast growing literature in the direction that we are envisioning, namely, to develop a theory of manifold learning in Wasserstein space.
Most studies so far explore the tangent space of the Wasserstein manifold at a suitably selected reference point (e.g., the barycenter of the dataset \cite{AC2011}) so as to take advantage of the Hilbert structure of the tangent space and identify the dominant modes of variation. Such an approach typically leads to non-smooth and non-convex optimization~\cite{cazelles2018geodesic}.

Herein, we follow a different route in which we place a probability measure on the sought structures (manifolds, modes, interpolating paths that we wish to learn). To illustrate the effect of a probability law on paths, Figure \ref{sample_path} highlights the flow of Gaussian one-time marginals along linear paths that differ in the joint distribution between the corresponding end-point variables. Our approach has several advantages. First, it enjoys a multi-marginal optimal transport~\cite{pass2015multi} formulation and the consequent computational benefits (solvable via linear programming). Second, it provides varying likelihood in the space of admissible paths/curves. When all the target distributions (observations) are Dirac measures, the problem reduces to interpolation in Euclidean space.

\begin{figure}[!htb]  
\centering
\subfigure[Positive correlation]{\resizebox{!}{4cm}{\includegraphics{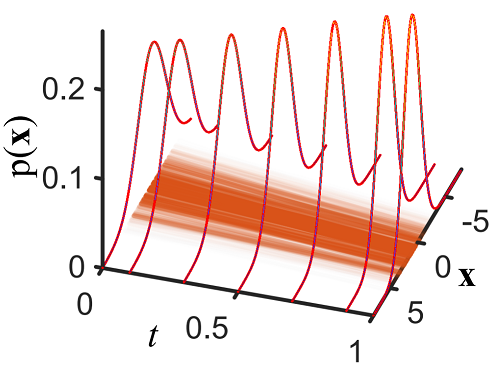}}}
\subfigure[Negative correlation]{\resizebox{!}{4cm}{\includegraphics{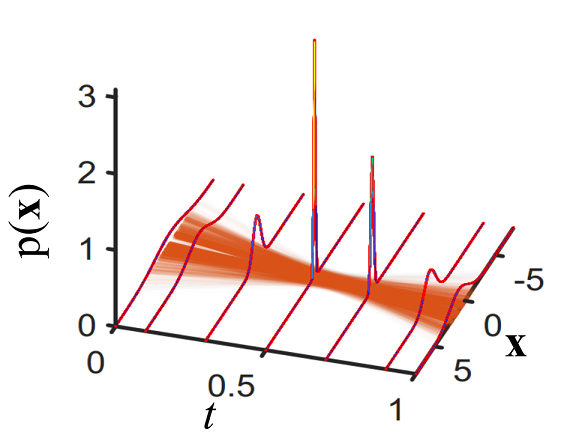}}}
\caption{Marginals along linear paths (lines) with identical end-point distributions and different probability laws (positive and negative correlation between end-point marginals); the intensity of color is proportional to the likelihood of the path.}
\label{sample_path}
\end{figure}

The first attempt to developing principal component analysis (PCA) for probability measures in)~\cite{bigot2017geodesic} utilized the pseudo-Riemannian structure of $(\PD,W_2)$ space~\cite{ambrosio2004gradient}, labeled geodesic PCA (GPCA). Unfortunately, implementation can be challenging even in the simplest case of probability measures supported on a bounded subset of the real line due to non-differentiable and non-convex cost and constraints. Other studies introduce suitable relaxation which however prevents solutions from relating to geodesics in Wasserstein space \cite{cazelles2018geodesic}, \cite{wang2013linear}, \cite{seguy2015principal}. Part of the ensuing problems in developing PCA in $(\PD,W_2)$ are due to the fact that this is a positively curved space \cite{ambrosio2008gradient}. Earlier relevant papers on regression in Wasserstein space include \cite{jiang2011geometric,fletcher2013geodesic}. 

The model of geodesics in $(\PD,W_2)$ as the learning hypothesis comes with an important caveat. At times, it may lead to underfitting in that the hypothesis class, $W_2$-geodesics, i.e., $(\PD,W_2)$-line segments, as opposed to higher order $(\PD,W_2)$-curves \cite{chen2018measure,benamou2018second}) is too restrictive. This may result in temporally distant points to be more strongly correlated than adjacent ones.

In this paper the emphasis is placed on the path space of linear segments from $[0,1]$ to $\Rd$ ($\bol{L}([0, 1],\Rd)$), 
and the problem of interpolating time-indexed target distributions which is cast as seeking a suitable probability law on line segments; it is seen as linear regression lifted to the space of probability distributions. We will discuss how the setting extends to the case where data/distributions are provided with no time-stamp. Interpolation with higher-order curves in $(\PD,W_2)$ that employs multi-marginal formulation is a possible future direction, as well as entropic regularization for computational purposes; these last topics will be noted in \cite{inpreparation}.
\section{Background and notation}
\subsection{Notation}
We denote by $(\PD,W_2)$ the Wasserstein space where $\PD$ is the set of Borel probability measures with finite second moments, and $W_2$ is the Wasserstein distance (whose definition will be recalled later on).
If a measure $\mu\in \PD$ is absolutely continuous with respect to the Lebesgue measure on $\Rd$, 
we will represent by $\mu$ both the measure itself and its density. 
The push-forward of a measure $\mu$ by the measurable map $T : \Rd \to \Rd$ is denoted by $T_\#\mu = \nu \in \mathcal P(\Rd)$, meaning
$\nu(B) = \mu(T^{-1}(B))$ for every Borel set $B$ in $\Rd$. The Dirac measure at point $\x$ is represented by $\delta_{\x}$. We denote the set of linear paths from the time interval $[0,1]$ to the state space $\Rd$ by $\Omega=\bol{L}([0, 1],\Rd)$. There is a bijective correspondence $(X_{0,1})$ between $\Omega$ and $\Rd \times \Rd$ using the values of each line at $t=0$ and $t=1$ such that any $\omega=(\omega_t)_{t\in[0,1]}\in\Omega$ corresponds to  $X_{0,1}(\omega):=(
\x_0,\x_1)$ for the endpoints $\x_0=\omega_0, \x_1=\omega_1 \in \Rd$. We equip $\Omega$ with the canonical  $\sigma$-algebra generated by the
projection map $X_{0,1}$. 
We often write $\mathbf{R}^{2d}:=\Rd \times \Rd$ as a shorthand.
For any $\pi(\x_0,\x_1)\in \PDD$, the one-time marginals are $\mu_t:=((1-t)\x_0+t\x_1)_\#\pi,~t\in[0,1]$.

\subsection{Background on Optimal Transport Theory}
We now provide a concise overview of optimal mass transportation theory based on  \cite{villani2003topics,ambrosio2013user}.

Let $\mu_0$ and $\mu_1$ be two probability measures in $\PD$. In the Monge formulation of optimal mass transportation, the cost
\begin{equation*}
  \int_{\Rd} ||T(\x)-\x||_2^2 ~\mu_0(d\x)   
\end{equation*}
is minimized over the space of maps
\begin{align*}
    T: \quad &\Rd \rightarrow \Rd\\ \nonumber
    &\x \mapsto T(\x)
\end{align*}
which ``transport'' mass $\mu_0(d\x)$ at $\x$ so as to match the final distribution, namely, $T_\#\mu_0 = \mu_1$. If $\mu_0$ and $\mu_1$ are absolutely continuous, it is known that the optimal transport problem has a unique solution which is the gradient of a convex function $\phi$, i.e., $T^*=\nabla\phi(\x)$.

This problem is highly nonlinear and, moreover, a transport map may not exist as is the case when $\mu_0$ is a discrete probability measure and $\mu_1$ is absolutely continuous.
To alleviate these issues, Kantorovich introduced a relaxation in which one seeks a joint distribution (coupling) $\pi$ on $\Rd\times\Rd$, with the marginals $\mu_0$ and $\mu_1$ along the two coordinates, namely,
\begin{equation*}
W_2^2(\mu_0, \mu_1):=    \hspace*{-10pt}\inf_{\pi \in \Pi(\mu_0, \mu_1)} \int_{\Rd\times\Rd} \hspace*{-5pt}||\bol{x-y}||^2 \pi(d\x d\y)
\end{equation*}
where $\Pi(\mu_0, \mu_1)$ is the space of couplings with marginals $\mu_0$ and $\mu_1$. 
In case the optimal transport map exists, we have $\pi = (\Id \times T^*)_\#\mu_0$,
where $\Id$ denotes the identity map. 

The square root of the optimal cost, namely $W_2(\mu_0, \mu_1)$, defines a metric on $\PD$ referred to as the Wasserstein metric \cite{ambrosio2004gradient,villani2008optimal}. Further, assuming that $T^*$ exists, the constant-speed geodesic between $\mu_0$ and $\mu_1$ is given by
\begin{equation*}
    \mu_t=\{(1-t)\x+tT^*(\x)\}_\#\mu_0,~~0\leq t \leq 1,
\end{equation*}
and is known as {\em displacement} or {\em McCann  geodesic}.

We now discuss two cases of transport problems that are of interest in this paper.
\subsubsection{Transport between Gaussian distributions}
In case $\mu_0 \sim \mathcal N(\bol{m}_0,\bol{\Sigma}_0)$ and $\mu_1 \sim \mathcal N(\bol{m}_1,\bol{\Sigma}_1)$, i.e., both Gaussian, the transport problem has a closed-form solution ~\cite{malago2018wasserstein}. Specifically,
\begin{equation*}T^*:\x\to {\scriptstyle\bol{\Sigma}_0^{-1/2}(\bol{\Sigma}_0^{1/2}\bol{\Sigma}_1\bol{\Sigma}_0^{1/2})^{1/2}\bol{\Sigma}_0^{-1/2}\x},
\end{equation*}
\begin{equation}
\label{Gaussian_W2}
    {\scriptstyle W_2(\mu_0,\mu_1)=\sqrt{||\bol{m}_0-\bol{m}_1||^2+\text{tr}(\bol{\Sigma}_0+\bol{\Sigma}_1-2 \bol{S})} }
\end{equation}
where $\text{tr}(.)$ stands for trace and
\begin{equation}
\label{cross_covariance_Gaussian}
    {\scriptstyle\bol{S}=(\bol{\Sigma}_0\bol{\Sigma}_1)^{1/2}=\bol{\Sigma}_0^{1/2}(\bol{\Sigma}_0^{1/2}\bol{\Sigma}_1\bol{\Sigma}_0^{1/2})^{1/2}\bol{\Sigma}_0^{-1/2}}.
\end{equation}
The McCann  geodesic $\mu_t$ corresponds to Gaussian distributions with mean $\bol{m}_t = (1-t)\bol{m}_0 +t\bol{m}_1$ and covariance
\begin{align}
\label{Geodesic_Gaussian}
    {\scriptstyle\bol{\Sigma}_{t}}&={\scriptstyle(1-t)^2\bol{\Sigma}_0+t^2\bol{\Sigma}_1+t(1-t)(\bol{\Sigma}_0\bol{\Sigma}_1)^{1/2}}\\ \nonumber
    &{\scriptstyle+t(1-t)(\bol{\Sigma}_1\bol{\Sigma}_0)^{1/2}}.
\end{align}

\subsubsection{Multi-marginal optimal transportation}
The multi-marginal problem is a generalization of optimal transport in which, for a set of given marginals, a correlating law is sought to minimize a cost function. This problem and its applications are surveyed in \cite{pass2015multi,nenna2016numerical}. The Kantorovich-version of this problem for  marginals $\{\mu_i\}_{i=1}^N$ is to minimize the cost
\begin{equation*}
\int_{\mathbf{R}^{Nd}} C(\x_1,...,\x_N)d\gamma(\x_1,...,\x_N)
\end{equation*}
where $\gamma(\x_1,...,\x_N)\in \mathcal{P}_2(\mathbf{R}^{Nd})$ such that ${\x_i}_{\#}\gamma=\mu_i$. It is a linear optimization problem over a weakly compact and convex set for which the numerical methods to solve it efficiently are well studied \cite{nenna2016numerical,benamou2019generalized}.

\section{Least-squares in $W_2$}
We start by highlighting the analogy with least squares in Euclidean setting.
Here, given a dataset $\{(t_i,\x_i)\}_{i=1}^N$, we seek a linear model
$
    \x=\x_0+t\bol{v}+\bm{\epsilon},
$
where $\x_0\in \Rd$ and $\bol{v}\in \Rd$ serve as the initial point and tangent vector, respectively, and $\bm{\epsilon}$ represents $\Rd$-valued measurement error. Least squares estimation provides a choice for $\x_0$ and $\bol{v}$ with the least variance among all linear unbiased estimators, assuming that errors are uncorrelated having zero mean and finite variance.

In our setting, the data points lie in the probability space $\PD$, and this entails a probabilistic representation of lines that generalizes linear regression to $(\PD,W_2)$ as discussed below. 
\begin{problem}
Given $\{\mu_{t_i}\}_{i=1}^{N}\in \PD$ a family of probability measures indexed by timestamps $\{t_i\}_{i=1}^{N}$, determine  
\be\label{eq:primal}
\inf_{(g_t)_{t\in [0,1]}\in \mathcal{G_R}} \sum_{i=1}^{N}\frac{1}{N} W_2^2(g_{t_i},\mu_{t_i})
\ee
where 
$
\mathcal{G_R}=\{ 
{\scriptstyle g_t=((1-t)\mathbf{x}_0+t\mathbf{x}_1)_\#\pi \mid \pi(\mathbf{x}_0,\mathbf{x}_1)\in \mathcal{P}_2(\mathbf{R}^{2d})}\}.
$

\end{problem}
Naturally, $g_t$, for $t\in [0,1]$, represent one-time marginals. Also, $\mathcal{G_R}$ includes as a subset the set of all geodesics on $\PD$. Indeed, any $\pi(\mathbf{x}_0,\mathbf{x}_1)$ which is an optimal coupling between two arbitrary distributions on $\PD$, matches a geodesic on $\PD$. Therefore, (\ref{eq:primal}) seeks the least sum of squared distances over a set that contains constant-speed geodesics.

The search space, $\mathcal{G_R}$, is the building block of this study which comes with noteworthy properties. First, it leads to a multi-marginal transport problem (Section \ref{Multi_marginal_sec}). Computationally, this is highly efficient 
as it amounts to a linear programming problem. Besides, entropic regularization can be used to considerable advantage (owing to the application of Sinkhorn's algorithm).

The problem amounts to finding a probability measure over the space of lines which defines a flux over linear paths for $t\in [0,1]$. 
It should be noted that $t$ is taken to lie within $[0,1]$ for simplicity, however, we can extend the range of $t$ to any arbitrary interval. The obtained curve in $\PD$ can be used for prediction as one can extrapolate to values of $t$ beyond the range in the observed data-set. 

We first address the absolute continuity of the elements in $\mathcal{G_R}$, a property of ``smooth'' curves in $(\PD,W_2)$ recalled below.

\begin{defi}[\cite{ambrosio2008gradient}]
Let $\mu_t: (0, 1) \rightarrow \PD$ be a one-parameter family in $\PD$; we say that $(\mu_t)_{t\in[0,1]}$ belongs to $AC(0, 1; \PD )$, if there exists $m \in L^1(0, 1)$ such that
$$W_2(\mu_s, \mu_t) \leq \int_{s}^t m(r) dr, \quad \forall ~0 < s \leq t < 1.
$$
\end{defi}

Absolutely continuous curves are important due to the boundedness of their metric derivative, namely $\mu_t'$, by $m(t)$~\cite{ambrosio2008gradient}, i.e.,
$$\mu_t':= \lim_{s\rightarrow t}{\frac{W_2(\mu_s, \mu_t)}{|s-t|}}\leq m(t).$$
The next theorem establishes absolute continuity of elements in $\mathcal{G_R}$.
\begin{theo}
For any $(g_t)_{t\in[0,1]}\in \mathcal{G_R}$ where $g_t=((1-t)\mathbf{x}_0+t\mathbf{x}_1)_\#\pi,~ t\in[0,1]$, we have $(g_t)_{t\in[0,1]} \in \text{AC}(0,1;\PD)$.
\end{theo}
\begin{proof} For $0 < s \leq t < 1$ , $W_2^2(g_s,g_t)$ is bounded by
$
  (t-s)^2 c 
   $ where $c$ is the bound derived from the finiteness of the second moments of $\pi$.
Therefore,
$W_2(g_s,g_t)\leq (t-s)\sqrt{c}=\int_{s}^t \sqrt{c} dr$.
\end{proof}

A natural question arises as to whether  
\[
F(\pi):=\sum_{i=1}^{N}\frac{1}{N} W_2^2(((1-t_i)\mathbf{x}_0+t_i\mathbf{x}_1)_\#\pi,\mu_{t_i}),
\]
originating in (\ref{eq:primal}), is strictly displacement convex
with respect to $\pi$;
a functional on $\PD$ is said to be (strictly) displacement convex if it is (strictly) convex along geodesics on $\PD$ in the classical sense~\cite{mccann1997convexity,villani2003topics}.
This notion of convexity inherits properties of classical convex analysis, in particular, strict displacement convexity guarantees uniqueness of the minimizer.

Unfortunately, the following counter-example reveals that $F(\pi)$ may not be displacement convex, in general.
Consider $\pi_0 := \ud\delta_{(0,1)} + \ud\delta_{(3,2)}$ and $\pi_1 := \ud\delta_{(0,0)} + \ud\delta_{(-3,2)}$, and a single data point
$\mu_{\{t=1\}} = \ud \delta_{1} + \ud \delta_{2}$ at $t=1$. The unique optimal map which pushes forward $\pi_0$ to $\pi_1$, maps $(0, 1)$ into $(-3, 2)$ and $(3, 2)$ into $(0, 0)$, and thus the displacement interpolation $\pi_s$ between these two measures is $\pi_s = \ud\delta_{(-3s,s+1)} + \ud\delta_{(3-3s,2-2s)}, ~s \in [0, 1]$.
Computing $F(\pi_s)$ along $(\pi_s)_{0\leq s\leq 1}$ gives 
\begin{align*}
F(\pi_s)= &  W_2^2({\mathbf{x}_1}_\#\pi_s,\mu_{\{t_i=1\}})\\
=&\min(\frac{5}{2}s^2,\frac{5}{2}s^2-3s+1),
\end{align*}  
which fails to be convex in $s$.

Thus, in the next section, we provide a multi-marginal optimal transport formulation for (\ref{eq:primal})  which helps circumvent issues with the inherent non-linearity and provides a handle on solutions.

\section{Multi-marginal reformulation} \label{Multi_marginal_sec}



The following gives a multi-marginal formulation of \eqref{eq:primal} and establishes properties of the minimizer.

\begin{theo}\label{theo:mm} Problem (\ref{eq:primal}) assumes the following multimarginal formulation
\begin{align}
& {\scriptstyle\inf_{\substack{\pi}} F(\pi)= \scriptstyle\inf_{\substack{\gamma}} 
\int \sum_{i=1}^{N}\frac{1}{N}\nonumber ||(1-t_i)\x_0+t_i\x_1-\y_i||^2d\gamma }\\
& {\scriptstyle \quad \quad \quad \text{subject to}\quad {\y_i}_\#\gamma=\mu_{t_i}(\y_i), \forall i=1,...,N}.  \label{main_eq0}
\end{align}
where $\gamma(\x_0,\x_1,\y_1,...,\y_N)\in \mathcal{P}_2(\mathbf{R}^{(N+2)d})$
and $\pi(\x_0,\x_1)\in \mathcal{P}_2(\mathbf{R}^{2d})$. Moreover, the minimizer for $F(\cdot)$ exists and is  $\pi^{*}=(\x_0,\x_1)_\#\gamma^{*}$ where $\gamma^{*}$ is the minimizers of the right-hand side in \eqref{main_eq0}.
\end{theo}




\begin{proof}
We sketch the main steps of the proof and we refer to \cite{inpreparation} for details. 
First, suppose $\pi(\x_0,\x_1)\in \mathcal{P}_2(\mathbf{R}^{2d})$ and $\mu(\y)\in \PD$ are given such that $g_t=\{(1-t)\x_0+t\x_1\}_\#\pi,~t\in[0,1]$ and $\eta(\x_0,\x_1,\y)\in \Pi(\pi,\mu)$, namely, a coupling between $\pi$ and $\mu$. We define
\begin{equation*}
    W_{\eta}(g_{t},\mu):=\int_{\mathbf{R}^{3d}} ||(1-t)\x_0+t\x_1-\y||^2d\eta.
\end{equation*}
It is easy to check that $W_{2}^{2}(g_{t},\mu) \leq W_{\eta}^2(g_{t},\mu),~\forall t\in [0,1].$ We can show the tightness of this inequality for some $\eta$, namely, $W_{2}^{2}(g_{t},\mu) =\min_{\eta(\x_0,\x_1,\y)\in \Pi(\pi,\mu)} W_{\eta}^2(g_{t},\mu)$.
This can be shown by the application of Gluing Lemma  (see Proposition 7.3.1 in \cite{ambrosio2008gradient}). Using the Disintegration Theorem~\cite{ambrosio2008gradient}, we can extend this result to a family of measures $\{\mu_{t_i}(\y_i)\}_{i=1,...,N}\in \PD$ to show that 
\begin{align*}\scriptstyle
\sum_{i=1}^N W_{2}^{2}(g_{t_i},\mu_{t_i}) =\min_{\substack{\gamma\in \mathcal{P}_2(\mathbf{R}^{(N+2)d})}} &\scriptstyle\sum_{i=1}^N W_{\gamma}^2(g_{t_i},\mu_{t_i}) \nonumber \\
&\scriptstyle\hspace*{-4.3cm}\text{subject to}\quad {\y_i}_\#\gamma=\mu_{t_i}(\y_i), 
 (\x_0,\x_1)_\#\gamma=\pi.
\end{align*}
Let $\gamma^*$ be any minimizer of the right-hand side term in~(\ref{main_eq0}) which exists due to the existence of solution to the multi-marginal problem with quadratic cost~\cite{chen2018measure}. It follows that $\pi^*=(\x_0,\x_1)\#\gamma^*$ minimizes left-hand side of (\ref{main_eq0}). 
\end{proof}

\begin{rem}\rm 
In Lagrangian particle tracking, one can track the Lagranigian particles resident in a fluid to measure the desired properties of them at several time steps. This tracking technique results in not only distributions of particles but also {\em correlation between different points in time}. In the proposed approach, the correlation of data can be incorporated by adding the linear constraints $(\y_i,\y_j)_{\#}\gamma=\pi_{t_it_j}$, where $\pi_{t_it_j}$ represents the joint at time steps $t_i$ and $t_j$. 
\end{rem}

\begin{rem}\rm 
The linear problem in multi-marginal formulation (\ref{main_eq0}) is  computationally burdensome. A well-known strategy to alleviate this is to regularize the cost by entropy which makes the problem strictly convex. Then, we can employ iterative Bergman projections~\cite{benamou2015iterative} to solve the regularized problem.  
\end{rem}
 The following proposition accounts for  consistency of the proposed method with linear regression in a Euclidean space when the target distributions are Dirac measures.
\begin{prop}
If all data points are Dirac measures, i.e., $\mu_{t_i}(\y_i)=\delta_{\mathbf{v}_i}(\y_i),~i=1,...,N$ where $\mathbf{v}_i\in \Rd$, the multi-marginal formulation in (\ref{main_eq0}) reduces to Euclidean regression. 
\end{prop}
\begin{proof}
Any $\gamma$ that satisfies the marginal constraints in (\ref{main_eq0}) can be written as $\gamma=\pi(\x_0,\x_1)\otimes \delta_{\mathbf{v}_1}(\y_1)\otimes...\otimes \delta_{\mathbf{v}_N}(\y_N)$, where $\pi(\x_0,\x_1)\in \PDD$ and $\otimes$ represents the product of two measures. Using this fact, we can show that
\begin{align*}
\label{discrete_consistency}
&\scriptstyle\inf_{\substack{\gamma}}  \int_{\mathbf{R}^{(N+2)d}} \sum_{i=1}^{N}\frac{1}{N} ||(1-t_i)\x_0+t_i\x_1-\y_i||^2d\gamma\nonumber \\
&\scriptstyle=\inf_{\substack{\mathbf{w}_0,\mathbf{w}_1\in \Rd}} \sum_{i=1}^{N}\frac{1}{N} ||(1-t_i)\mathbf{w}_0+t_i\mathbf{w}_1-\mathbf{v}_i||^2.
\end{align*}
\end{proof}

\subsection{Measure-valued PCA}
Classical PCA produces subspaces of a given dimension that approximate a data point cloud. 
Their projection onto these subspaces can serve as approximate coordinates for the data points.  
Our aim is to extend the concept so as to approximate and parametrize families of probability measures~$\{\mu_{i}\}_{i=1}^N$.  In this context, the first principal component corresponds to a line segment in $\PD$ that is at a minimal distance from the data points which are now probability distributions. To this end, it is natural to consider the following problem
\begin{equation*}
\hspace*{-3pt}\inf_{{\scriptscriptstyle \pi\in \PDD}} \hspace*{-2pt}\sum_{i=1}^{N}\frac{1}{N} \min_{t\in \mathbf{R}} {W_2^2(((1-t)\x_0+t\x)}_{\#}\pi,\mu_{i}).
\end{equation*}
An equivalent multi-marginal formulation is
\begin{align}
\label{dim_red}
&\inf_{\substack{\gamma}}  \sum_{i=1}^{N} \hspace*{-4pt}\frac{1}{N} \hspace*{-1pt}\min_{t\in \mathbf{R}} \int_{\mathbf{R}^{(N+2)d}} \hspace*{-7pt} ||(1-t)\x_0+t\x_1-\y_i||^2d\gamma \nonumber \\
&\quad \quad \quad \text{subject to}\quad {\y_i}_\#\gamma=\mu_{i}(\y_i), \forall i=1,...,N 
\end{align}
where $\gamma(\x_0,\x_1,\y_1,...,\y_N)\in \mathcal{P}_2(\mathbf{R}^{(N+2)d})$.
However, in the present setting of $\PD$, the problem is non-convex.
We employ coordinate descent method~\cite{wright2015coordinate} in seeking a minimizer, although global optimality is not guaranteed. 

The key idea is to successively minimize for $t$ and $\gamma(\x_0,\x_1,\y_1,...,\y_N)$ to find the minimum of (\ref{dim_red}). To do so, at each iteration, we fix either $\gamma$ or the parameters $t$, and optimize in the other direction. With respect to $t$ the problem is quadratic assuming the closed-form solution
\begin{equation}
    \label{update_t}
    t^*_i=\frac{\int_{\mathbf{R}^{3d}}  <\y_i-\x_0,\x_1-\x_0> d\eta_i(\x_0,\x_1,\y_i)}{\int_{\mathbf{R}^{2d}}  ||\x_1-\x_0||^2 d\pi_i(\x_0,\x_1)}
\end{equation}
where $\eta_i(\x_0,\x_1,\y_i)=(\x_0,\x_1,\y_i)_{\#}\gamma$ and $<.,.>$ denotes a corresponding inner product. 
Minimization over $\gamma$ is a linear programming problem.

\section{Results for Gaussian measures}
In this section, we consider Gaussian target distributions, i.e., $\mu_{t_i}(\y_i)\sim N(\bol{m}_{\y_i},\bol{C}_{\y_i})$. By applying the proposed method to this problem, we will show that the optimal interpolating curve remains Gaussian at each time step. Regarding (\ref{Gaussian_W2}), we observe that in  regression analysis for Gaussian measures, the means can be treated separately by linear regression in Euclidean space. Thereby, the means for the optimal curve in $(\PD,W_2)$ can be readily found. Therefore, without loss of generality, we assume the means of $\mu_{t_i}$'s to be zero. The following proposition recast the problem as an SDP.
\begin{prop}
Consider Gaussian-points $\mu_{i}\sim N(0,\bol{C}_{\y_i})$. The minimizing $\gamma^*(\x_0,\x_1,\y_1,...\y_N)$ in \eqref{dim_red}
 is a Gaussian distribution with mean zero and covariance that solves
\begin{align}
\label{SDP}
   {\scriptstyle\min_{\mathbf{C}_{\gamma}\succeq 0} }
   &{\scriptstyle\text{tr}((1-2\overline{t}+\overline{t}^2)\bol{C}_{\x_0}+\overline{t^2}\bol{C}_{\x_1}+ 2(\overline{t}-\overline{t^2})\bol{S}_{{\x_0}{\x_1}})} \nonumber\\
   &{\scriptstyle-\frac{2}{N}\sum_{i=1}^{N} \text{tr}((1-t_i)\bol{S}_{{\x_0}{\y_i}}+t_i\bol{S}_{{\x_1}{\y_i}})} \nonumber\\
\end{align}
where $\overline{t}=1/N\sum_{i=1}^{N}t_i$, $\overline{t^2}=1/N\sum_{i=1}^{N}t_i^2$ and
\begin{equation}
\label{Covariance_Matrix}
\mathbf{C}_{\gamma}=
\left[
    \begin{smallmatrix}
    \bol{C}_{\x_0}       & \bol{S}_{{\x_0}{\x_1}} & \bol{S}_{{\x_0}{\y_1}} & \dots & \bol{S}_{{\x_0}{\y_N}} \\
    \bol{S}_{{\x_0}{\x_1}}^T       & \bol{C}_{\x_1} & \bol{S}_{{\x_1}{\y_1}} & \dots & \bol{S}_{{\x_1}{\y_N}} \\
    \bol{S}_{{\x_0}{\y_1}}^T       & \bol{S}_{{\x_1}{\y_1}}^T & \bol{C}_{\y_1} & \dots & \bol{S}_{{\y_1}{\y_N}} \\
    \vdots & \vdots & \vdots & \ddots \\
    \bol{S}_{{\x_0}{\y_N}}^T       & \bol{S}_{{\x_1}{\y_N}}^T & \bol{S}_{{\y_1}{\y_N}} & \dots & \bol{C}_{\y_N}
\end{smallmatrix} \right].
\end{equation}

\end{prop}
\begin{proof}
When the marginals (${\{\mu_{t_i}\}}_{i=1}^N$) of the joint measure $\gamma$ in \eqref{main_eq0} are Gaussian, we can conclude that $\gamma^*$ in \eqref{main_eq0} is also Gaussian since the marginal constraints act only on the second moments of $\gamma$, namely, no constraint is imposed onto higher-order moments, and the cost function is quadratic in $\x_0,\x_1,\y_1,...,\y_N$.
Simple calculation shows the cost function in (\ref{main_eq0}) for any Gaussian $\gamma$ with covariance matrix given in (\ref{Covariance_Matrix}) reads
\begin{align*}
&\scriptstyle{\int_{\mathbf{R}^{(N+2)d}} \sum_{i=1}^{N}\frac{1}{N} ||(1-t_i)\x_0+t_i\x_1-\y_i||^2d\gamma =}\\ \nonumber
&(\scriptstyle{1-2\overline{t}+\overline{t}^2)\text{tr}(\bol{C}_{\x_0})+\overline{t^2}\text{tr}(\bol{C}_{\x_1})+2(\overline{t}-\overline{t^2})\text{tr}(\bol{S}_{{\x_0}{\x_1}})}\\ \nonumber
   &\scriptstyle{-\frac{2}{N}\sum_{i=1}^{N} \text{tr}((1-t_i)\bol{S}_{{\x_0}{\y_i}}+t_i\bol{S}_{{\x_1}{\y_i}})}.
\end{align*}
 From this, we can readily conclude that the multi-marginal formulation in (\ref{main_eq0}) is equivalent to the semi-definite programming in (\ref{SDP}).
\end{proof}

Noticing that $(\x_0,\x_1)_\#\gamma=\pi$, one can write the interpolating curve as
\begin{equation}
\label{interp_curve_relaxed}
   \scriptstyle{ \mu_t \sim N(0, (1-t)^2\bol{C}^*_{\x_0}+t^2\bol{C}^*_{\x_1}+t(1-t)(\bol{S}^*_{{\x_0}{\x}}+{\bol{S}^*_{{\x_0}{\x}}}^T))}
\end{equation}
where asterisk in the equation above denotes the optimal values obtained in (\ref{SDP}). 
One can compare the result above to (\ref{Geodesic_Gaussian}) by noticing that in \eqref{interp_curve_relaxed} the term $\bol{S}^*_{{\x_0}{\x_1}}$ is not equal to $(\bol{C}^*_{\x_0}\bol{C}^*_{\x_1})^{1/2}$. 
Therefore, the interpolating curve ($\mu_t$) in (\ref{interp_curve_relaxed}) doesn't characterize a geodesic in $(\PD,W_2)$, which may underfit the dataset as discussed earlier.   

We conclude with numerical examples of density curves that interpolate a set of given Gaussian
marginals. To do so, first we consider 30 different marginals to be one-dimensional Gaussian distributions with zero
mean (blue curves in Fig. \ref{Gaussian_example}), where the focus is on interpolation of the respective variances.  Then, the SDP in (\ref{SDP}) is solved to obtain the optimal joint distributions $\gamma$ and $\pi(\x_0,\x_1)=(\x_0,\x_1)_{\#}\gamma$. Figure \ref{MV_reg} depicts the density curve for different values of $t$ along with the dataset. Furthermore, to compare the result to geodesic regression, in Fig. \ref{Geo_reg} the optimal geodesic in $(\PD,W_2)$, which interpolates the dataset, is represented. One can notice that the measure-valued regression captures the variation in the dataset better than geodesic one. This ensues from the fact that in geodesic regression an interpolating curve in $(\PD,W_2)$ with highest correlated end-points is sought. However, in the framework of this paper, this constraint is relaxed which can also moderate underfitting.


\begin{figure}[!tb]  
\centering
\subfigure[Measure-valued regression]{\label{MV_reg}\resizebox{!}{5cm}{\includegraphics{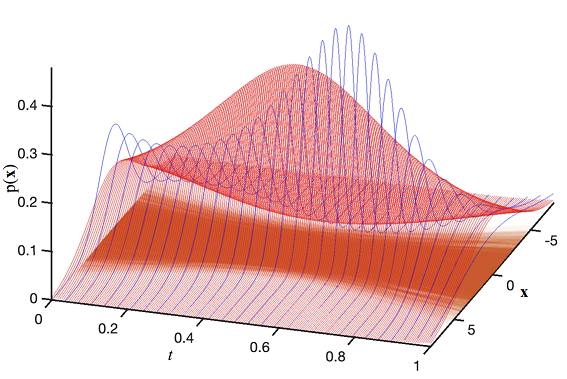}}}
\subfigure[Geodesic regression]{\label{Geo_reg}\resizebox{!}{5cm}{\includegraphics{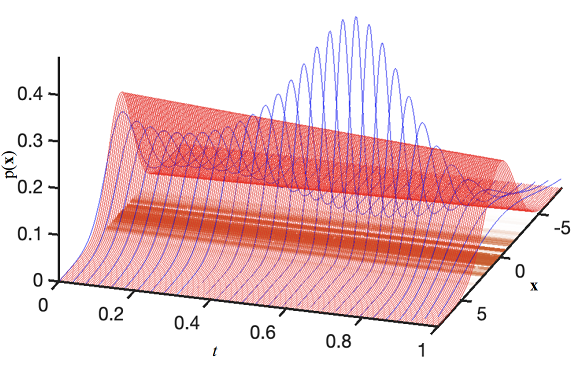}}}
\caption{Interpolation of one-dimensional Gaussian marginals. Blue curves are the given distributions and red ones are the interpolation. The intensity of color in lines is proportional to the likelihood of each path.}
\label{Gaussian_example}
\end{figure}



\addtolength{\textheight}{-3cm}   


\section{CONCLUSIONS AND FUTURE WORKS}
In this paper we introduce an approach for statistical learning in the space of distributions, metrized by the Wasserstein metric, which can be regarded as a generalization of regression analysis and PCA to the space of probability measures. To alleviate the issues with non-linearity of this formulation, we capitalize on a link between this problem and multi-marginal transport.  In addition to computational benefits and consistency with Euclidean regression, our model can readily incorporate correlations between distributions (i.e., available higher-dimensionality marginals) in the dataset. The potential of the theory is envisaged in aggregate data inference with application to metapopulation dynamics~\cite{nichols2017using} where the identity of individuals is not available,  particle \& power spectra  tracking~\cite{jiang2011geometric}, system identification~\cite{ljung2010perspectives}, and measurement noise  and  uncertainty treatment in dynamical systems. A natural future research direction is towards utilizing higher-order curves in ($\PD,W_2$) along a dynamic formulation analogous to the proposed method.



\spacingset{.97}
\bibliographystyle{IEEEtran}
\bibliography{Reference_final}

\end{document}